\documentclass[a4paper,oneside,reqno,12pt]{amsart}
\usepackage{stackrel}
\usepackage{dsfont}
\usepackage{amsmath}
\usepackage{enumitem}
\usepackage{amsthm}
\usepackage{fullpage}
\usepackage{xcolor}
\usepackage{mathrsfs}
\usepackage{graphicx}
\usepackage{amssymb}
\usepackage{csquotes}
\usepackage[colorlinks]{hyperref}
\usepackage{cleveref}
\everymath{\displaystyle}
\usepackage{graphics}

\newcommand{\ep}{\varepsilon}

\def\ds{\displaystyle}

\def\p{\partial}

\def\n{\nabla}

\def\R{\mathbb{R}}

\def\h{\mathbb{H}}

\def\He1{\boldsymbol{H}^1_{\Gamma_b}(\Omega_{\ep})}
\def\H1{\boldsymbol{H}^1(\Omega_{\ep})}

\newtheorem{thm}{ \bf Theorem}
 
 \newtheorem{lem}[thm]{ \bf Lemma}
 \newtheorem{prop}[thm]{ \bf Proposition}
 \newtheorem{defn}[thm]{ \bf Definition}
 \newtheorem{rem}[thm]{ \bf Remark}

\usepackage[symbol]{footmisc}

\date{}
\title{Unfolding operator on Heisenberg Group and  applications in Homogenization}

\begin{document}

\author{A. K. Nandakumaran}
\address{Department of Mathematics, Indian Institute of Science, Bangalore, India 560012} 
\email{nands@iisc.ac.in}

\author{A. Sufian}
\address{Department of Mathematics, Indian Institute of Science, Bangalore, India 560012}
\email{abusufian@iisc.ac.in}

\begin{abstract}
The periodic unfolding method is one of the latest tool after multi-scale convergence to study multi-scale problems like homogenization problems. It provides a good understanding of various micro scales involved in the problem which can be conveniently and easily applied to get the asymptotic limit.
In this article, we develop {\it the periodic unfolding}  for the Heisenberg group which has a non-commutative group structure. In order to do this, the concept of greatest integer part, fractional part for the  Heisenberg group have been introduced corresponding to the periodic cell.  Analogous to the Euclidean unfolding operator, we  prove the  integral equality, $L^2$-weak compactness, unfolding gradient convergence and other related properties. Moreover, we have the adjoint operator for the unfolding operator which can be recognized as an average operator. As an application of the unfolding operator, we have homogenized the standard elliptic PDE with oscillating coefficients. We have also considered an optimal control problem and characterized the interior periodic optimal control in terms of unfolding operator.      
\end{abstract}



\maketitle

\section{Introduction}

The mathematical theory of homogenization was introduced in the 1970s in order to describe the behavior of composite materials.  Since then, several homogenization methods have been developed. Among many methods developed in the last 50 years two-scale convergence and unfolding methods are very effective techniques. The two-scale convergence was introduced by Nguetseng \cite{GN89} and later developed by Allaire in \cite{GA92} which has been extensively applied by various authors over the last few decades. There are plenty of methods available for the Euclidean setting or, more precisely, in the commutative group structure. The research  on Homogenization in non-commutative group structure is very limited.  Among the early results on the homogenization in non-commutative group structure, we cite the results by Biroli, Mosco, and Tchou \cite{MUNA97}. In this paper, the authors construct explicitly a periodic tilling associated with the Laplace operator $\Delta_\h$ associated with the Heisenberg group. They have analyzed the asymptotic behavior of its eigenfunctions in a domain with isolated Heisenberg periodic holes with Dirichlet boundary conditions on their boundaries. To establish the convergence to the homogenized problem, they employ Tartar’s energy method.  Another piece of work on homogenization in the Heisenberg group is due to Biroli, Tchou and Zhikov in \cite{MNAZ99}. The  problem has  revisited in  \cite{FM02} with less regular holes. Due to less regularity on the hole, they could not employ the method as in \cite{FM02}, and they used the method introduced in \cite{zhik96} by Zhikov.  For further reading, we refer to the articles \cite{FM06,FCT18} and references therein.       

\par Now, coming back to the Euclidean setting, the method of two-scale convergence in $\R^n$ is deeply related to the group structure of $\R^n$ and the definition of the periodic function in terms of group translation.   
The concept of the tiling and the periodic function defined in \cite{MUNA97} for the Heisenberg group, motivates B. Franchi and M. C. Tesi in \cite{FM02}
 to define the concept of two-scale convergence in the Heisenberg Group. As an application of this two-scale convergence, they have investigated a Dirichlet problem for a generalized Kohn Laplacian operator with strongly oscillating Heisenberg-periodic coefficients in a domain that is perforated by interconnected Heisenberg-periodic pipes. They have proved all the similar results as in Euclidean two-scale scale convergence.
\\
One of the latest methods for homogenization is {\it the periodic unfolding method} introduced by  Cioranescu, Damlamian, and Griso in \cite{priun}, where the micro scale is introduced at the micro level of the problem before taking the limit, whereas in two-scale convergence, the micro scale is recovered at the limit.  The unfolding operator is also quite  easy to apply in multi-scale analysis and help to see more deeply the microscopic scale. For the sake of the reader, we recall the two-scale convergence and unfolding operators in the  Euclidean space set-up. Let us recall the definition two-scale convergence and unfolding operator for the  Euclidean domain and will see how they are related to each other.
\par Let $Y$ be the reference cell and $\Omega$ be a bounded domain of $\R^n.$ The smooth $Y$ periodic (Euclidean sense) function space is denoted by $C_{\#}^{\infty}(Y).$
\begin{defn}[Two-scale convergence]
A family of function $\{u_{\varepsilon}\}\in L^2(\Omega)$ is said to be two-scale converges to $u_0\in L^2(\Omega\times Y)$, if for any $\psi\in C_c^{\infty}(\Omega;C_{\#}^{\infty}(Y))$, we have 
\begin{align}\label{two-scale convergenc}
\lim_{\varepsilon\to 0}\int_{\Omega}u_\varepsilon(x)\psi\left(x,\frac{x}{\varepsilon}\right)\, dx
=\frac{1}{|Y|}\int_{\Omega\times Y}u_0(x,y)\psi(x,y)\, dxdy
\end{align}  
\end{defn}  

Let  $$E_{\varepsilon}=\{k\in \mathbb{Z}^n:~\varepsilon k+\varepsilon Y \subset \Omega\},\,\Omega_\varepsilon=\bigcup_{k\in E_\varepsilon}\{k\varepsilon+\varepsilon Y\},~\Lambda_\varepsilon =\Omega\backslash \Omega_\varepsilon ~\text{(see Figure \ref{unfol}).}$$ 
The greatest integer part and fractional part with respect to $Y$ are denote by $\left[\frac{x}{\varepsilon}\right]_Y$ and $\left\{\frac{x}{\varepsilon}\right\}_Y$ respectively. Note that the micro scale $y$ is given in the limit $u_0=u_0(x,y)$, $y\in Y$. We now introduce this scale at $\varepsilon$ level itself using the scale decomposition of the Euclidean space $\mathbb{R}^n$. We will later give appropriate scale decomposition of the Heisenberg group. For $x\in \mathbb{R}^n$, we can write the $\varepsilon-$scale  decomposition as 
\[x= \varepsilon\left(\left[\frac{x}{\varepsilon}\right]_Y+ \left\{\frac{x}{\varepsilon}\right\}_Y\right),\]
where $\left[\frac{x}{\varepsilon}\right]_Y$ and $\left\{\frac{x}{\varepsilon}\right\}_Y$ 
are the integer and fractional parts, respectively  with respect to the reference cell $Y$ (see Figure 2). We introduce the scale $y$ for varying $ \left\{\frac{x}{\varepsilon}\right\}_Y$  and we have the following definition.
\begin{defn}[Unfolding operator]
For $\phi$ Lebesgue-measurable real valued function on $\Omega$, the unfolding operator $T^\varepsilon$ is defined as follows:
\begin{align}
T^{\varepsilon}(\phi)(x,y)=\begin{cases}
\phi\left(\varepsilon\left[\frac{x}{\varepsilon}\right]_Y+\varepsilon  y\right) & \,\, \emph{for}~(x,y)\in \Omega_\varepsilon \times Y\\
0 \hspace{2cm}& \emph{for}~(x,y)\in \Lambda_\varepsilon\times Y.
\end{cases}
\end{align}
\end{defn}
\begin{figure}\label{unfol}
\includegraphics[scale=.5]{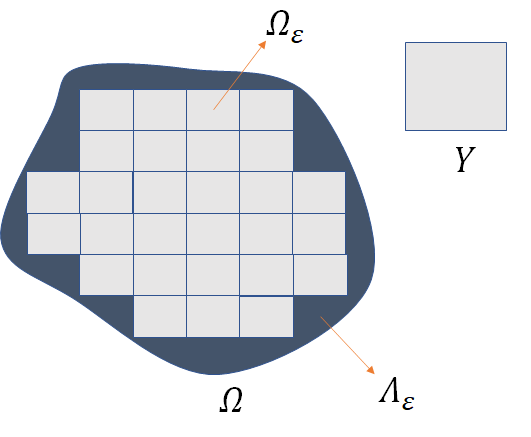}
\caption{Tiling of $\Omega$ in the Euclidean set up}
\end{figure}
The method of periodic unfolding is based on the concept of two-scale convergence.
The test functions used in two scale have  one macro scale $x$ which tells position in $\Omega$ another is micro-scale  $\frac{x}{\varepsilon}$ which tells position of $x$ in the reference cell. In unfolding this concept concept is used very explicitly. More-precisely, if we view our domain as $\Omega=\bigcup_{k\in E_\varepsilon}\{k\varepsilon+\varepsilon Y\}$ and $\psi \in C_c^{\infty}(\Omega;C_{\#}^\infty(Y)$.
Then, we can write \eqref{two-scale convergenc} as
\begin{align}&\lim_{\varepsilon\to 0}\int_{\Omega}u_\varepsilon(x)\psi^\varepsilon\left(x\right)\, dx\nonumber\\
&=\lim_{\varepsilon\to 0} \sum_{k}\int_{k\varepsilon+kY}u_\varepsilon(x)\psi\left(x,\frac{x}{\varepsilon}\right)\, dx \nonumber\\
&=\lim_{\varepsilon\to 0} \sum_{k}\int_{k\varepsilon+kY}u_\varepsilon(k \varepsilon+\varepsilon y)\psi\left(k \varepsilon+\varepsilon y,y\right)dy \nonumber\\
&=\frac{1}{|Y|}\lim_{\varepsilon\to 0} \int_{\Omega}\int_{Y}u_\varepsilon\left(\varepsilon\left[\frac{x}{\varepsilon}\right] +\varepsilon y\right)\psi\left( \varepsilon\left[\frac{x}{\varepsilon}\right]_Y+\varepsilon y,y\right)dy\nonumber\\
&=\frac{1}{|Y|}\lim_{\varepsilon\to 0}\int_{\Omega\times Y}T^{\varepsilon}(u_\varepsilon)(x,y)T^{\varepsilon}(\psi^\varepsilon)(x,y)\, dxdy\nonumber\\
&=\frac{1}{|Y|}\int_{\Omega\times Y}u_0(x,y)\psi(x,y)\, dxdy.\nonumber
\end{align}   
Observe that definition of two-scale convergence reduced to weakly convergence in\\ $L^2(\Omega\times Y)$ and it is easy to apply as it is  technicality less demanding. 
 \par There are some advantages of using this method; for example, while doing optimal control problems in periodic setup, the optimal control is easily characterized by the unfolding of the adjoint state, which helps to analyze asymptotic behavior see \cite{AAN19,semiopt,ARB,JDE,AN21}. This method reduces the definition of two-scale convergence in $L^p(\Omega)$ to weak convergence of the unfolding  sequence in $L^p(\Omega\times Y)$ for $1<p<\infty$. This is a very effective method in analyzing the various multi-scale problems; for details, see \cite{ smooth, CioDAmGri1} and references therein.

\par The unfolding method in $\R^n$ is intensely dependent  on the group structure of $\R^n$. we aim to develop a similar type of unfolding operator for the  Heisenberg group. As we have already mentioned that the concept of periodic functions and  tiling in Heisenberg group was introduced in \cite{MUNA97}, which motivates us to define the greatest integer part for $x\in \h^1$, that is $[x]_\h$ and fractional $\{x\}_\h.$ Using these definitions, we have defined unfolding operator $T^{\varepsilon}$ in the Heisenberg group. The definition of $T^{\varepsilon}$ for $\h^1$, keeps periodic function unchanged. As an application of this unfolding operator, we have considered a PDE $-\text{div}_\h(A^{\varepsilon}\nabla_\h )$ with Heisenberg-periodic oscillating coefficients in an open bounded domain $\Omega \subset \h^1.$ This model PDE is also considered in \cite{FM02}, in a perforated domain where they have used two-scale convergence to analyze the asymptotic behavior. Our aim is to introduce unfolding in Heisenberg group. Though we have applied it only to a standard problem, we hope to apply to more general problems and probably introduce unfolding to other non-commutative groups.
\par The rest of this article is organized as follows. In Section \ref{prelim}, we recall the definition and properties of periodic and non-periodic function spaces in Heisenberg group. In Section \ref{defofunfol}, definition of $[x]_\h$, $\{x\}_\h$, unfolding operators and adjoint operators are introduced. Properties and their proof are also given here. Finally, in Section \ref{Homgenization}, we  consider a  model PDE with oscillating coefficients, homogenize it and also shown the characterization of  the interior periodic optimal control for the interior periodic optimal control problem. We did not present the final homogenization of the optimal control problem as it follows along similar lines.

\section{preliminaries} \label{prelim}
Here, we introduce required notations which will be used through out the article and some preliminaries. We denote the $1$-dimensional Heisenberg group by $\h^1\cong \R^3$ and a typical point in $\h^1$ is denoted by $x=(x_1,x_2,x_3)$. For $p=(p_1,p_2,p_3), q=(q_1,q_2,q_3)\in \h^1$, the group operation is $$p\cdot q=(p_1+q_1,p_2+q_2,p_3+q_3+2(p_2q_1-p_1q_2).$$ The inverse of  $x\in \h^1$ is $x^{-1}=(-x_1,-x_2,-x_3).$ The family of non-isotropic dilations are denoted by $\delta_{\lambda}$ defined as
$$\delta_\lambda (x)=(\lambda x_1,\lambda x_2,\lambda^2 x_3)~~\text{for}~x\in \h^1 .$$
The left translation operator corresponding to $p\in \h^1$ denoted by $\tau_p$ defined as $$\tau_p(x)=p.x~~~\text{for}~~x\in \h^1.$$ We consider the following homogeneous norm with respect to $\delta$; for $x\in \h^1$ $$\|x\|_{\infty}:=max\left \{\sqrt{x_1^2+x_2^2},\,\,\sqrt{|x_3|}\right \}.$$ The associated distance between any $p,q\in \h^1$ given as $$d(p,q)=\|p^{-1}\cdot q \|_{\infty}.$$ There is a relation between this distance and Euclidean distance (see \cite{FM02}), which is stated in the following proposition.
\begin{prop}\label{equ-heidis}
 The function $d$ is a distance in $\h^1$. Further, it is homogeneous and left translation invariant, that is for any $p,q,x \in \h^1$ and $\lambda>0,$ $$ d(\delta_\lambda q,\delta_\lambda x)=\lambda d(q,x) \quad \text{and}\quad d(\tau_p q,\tau_p x)=d(q,x).$$
For any bounded subset $\Omega$ of $\h^1,$ there exist positive constant $c_1(\Omega),c_2(\Omega)$ such that $$c_1(\Omega)|p-q|_{\R^3}\leq d(p,q)\leq c_2(\Omega)\sqrt{|p-q|_{\R^3}}.$$ Here $|~\cdot~|_{\R^{3}}$ denotes the Euclidean norm.  
\end{prop}  

In particular the induced topology by $d$ and the Euclidean topology are coincide on $\h^1$. The usual Lebesgue measure is the left and right invariant Haar measure for the group. For any measurable set $S\subset \h^1,$ the Lebesgue measure of $S$ is denoted by $|S|.$  Because of the anisotropic dilation $\delta_\lambda$ for $\lambda>0,$ we have $|\delta_\lambda(S)|=\lambda^4 |S|.$ That is why the vector space dimension  of $\h^1$ is $3$, but the Hausdorff dimension is $4$.        
\par The Lie algebra of the left invariant vector field of $\h^1$ is given by
$$ X_1=\frac{\p }{\p x_1}+2x_2\frac{\p }{\p x_3},~~~
~X_2=\frac{\p}{\p x_2}-2x_1\frac{\p}{\p x_3},~~\text{and}~X_3=\frac{\p}{\p x_3}.$$
The only non-trivial commutator relation is,$$[X_1,X_2]=X_1X_2-X_2X_1=-4X_3.$$ 
The horizontal vector field H$\h^1$ is the span of the vector field $\{X_1,X_2\}.$ Hence, we will identify a section $\phi$ of H$\h^1$ with the function $\phi=(\phi_1,\phi_2):\h^1\to \R^2.$

\subsection*{Function Spaces}(see \cite{FM02}): 
Through out this article, $\Omega \subset \h^1$ is a bounded domain. For any integer $k>0$, $C^{k}(\Omega)$, $C^{\infty}(\Omega)$ denote the usual differentiable function spaces in the Euclidean sense. We denote by $C^k(\Omega;\text{H}\h^1),$ for $k\geq 0$, the set of all $C^k$ section of H$\h^1$. Now, we will define gradient and divergence as follows.
\begin{defn}
Let $f\in C^{1}(\Omega)$ and $\phi=(\phi_1,\phi_2)\in C^{1}(\Omega;\emph{H}\h^1)$ is a continuously differentiable section of \emph{H}$\h^1$, we define:
$$\n_{\h}f:=(X_1f,X_2f)~\,\,\emph{and}~\emph{div}_{\h}\phi=X_1 \phi_1+X_2\phi_2.$$
\end{defn}
\noindent Note that both $\n_{\h}$, div$_{\h}$ are left invariant differential operators. Also, $\n_\h f$ can be defined as a section of H$\h^1$ as
$$\n_\h f=(X_1f) X_1+(X_2f)X_2.$$
The Heisenberg gradient $\n_{\h}$ can be written in terms of Euclidean gradient $\n$ as $$\n_\h=C(x)\n, \quad \text{ where} \quad  C=C(x)=\begin{bmatrix}
1 & 0 & 2x_2\\
0 & 1 & -2x_1\end{bmatrix}.$$ Similarly, $\text{div}_\h\phi=\text{div} (C^t\phi),$ where $\text{div}$ is the Euclidean divergence in $\R^3$ and $C^t$ is the transpose of the matrix $C$. \\

\par \noindent For $1\leq p<\infty$, $L^{p}(\Omega)$ denotes the usual Euclidean $p$-integrable  space. Here, we will introduce all the necessary non-periodic function spaces
\begin{enumerate}[label=(\roman*)]
\item The set of all smooth sections of H$\h^1$ is denoted by $C^{\infty}(\Omega;\text{H}\h^1).$ Similarly the compactly supported smooth sections of H$\h^1$ is denoted $C^{\infty}_{c}(\Omega;\text{H}\h^1).$ 
\item Analogous to standard  Euclidean $H^1(\Omega)$ Sobolev space, we have the following Heisenberg Sobolev spaces$$H_{\h}^1(\Omega)=\{f\in L^{2}(\Omega):~X_1 f,X_2f \in L^{2}(\Omega)\}.$$ 
Further, $C^{\infty}(\Omega)\cap H^{1}_{\h}(\Omega)$ is dense in $H^{1}_{\h}(\Omega).$
\end{enumerate}
Through out this article, we will denote the cube $[0,2]^3$ by $Y$. We have this cube of side length $2$ instead of $[0,1]^3$ to avoid the intersection of tiles in the Heisenberg periodic setting. A set $G\subset \h^1$ is said to be $Y$-periodic if for any $x\in \h^1$ and $k\in \mathbb{Z}^3$, the translations $\tau_{2k}(x)\in G.$ The space $                                                                                                      \h^1$ is indeed $Y$-periodic. In this article, we will use $\h^1$ as $Y$-periodic set just like $\R^n=\biguplus_{k\in \mathbb{Z}^n} ([0,1)^n+k)$
\begin{enumerate}[label=(\roman*)]
\item \textit{Periodic function:} Let $f$ be a real valued function defined on $\h^1$. The function $f$ will be called $Y$-periodic if for any any $k\in \mathbb{Z}^3$,
$$f(\tau_{2k}(x))=f(x)~~\text{for all}~x\in \h^1.$$
A section $\phi$ in H$\h^1$ is called $Y$-periodic if the canonical co-ordinates are $Y$-periodic. 
\item We denote $C^{\infty}_{\#,\h}(Y)$, the space of smooth real valued $Y$-periodic functions.
\item For $1\leq p<\infty$, we denote by $L^{p}_{\#,\h}(Y)$, the space of $Y$- periodic functions such that $f|_{Y}\in L^{p}(Y)$ endowed with the norm $\|f\|_{L^p(Y)}.$
\item Similarly, $H^{1}_{\#,\h}(Y)$ denotes the space of all $f\in L^2_{\#,\h}$ such that $X_i f \in L^2(\delta_{\lambda}(Y))$ for all $\lambda>0$ endowed with the norm $\|f\|_{H^1_{\h}(Y)}$. This is a Hilbert space.
\end{enumerate}
We now introduce the periodic vector valued function spaces:
\begin{enumerate}[label=(\roman*)]
\item We denote $C^{\infty}_{c}(\Omega;C^{\infty}_{\#,\h}(Y))$, the space of all smooth functions on $\Omega\times \h^1$ such that for any $\phi \in C^{\infty}_{c}(\Omega;C^{\infty}_{\#,\h}(Y)),$ $x\to \phi(x,\cdot)$ is $C^{\infty}$ from $\Omega \to C^{\infty}_{\#,\h}(Y)$ with compact support.
\item The space of periodic smooth sections is denoted by $C^{\infty}_{\#,\h}(Y;\text{H}\h^1).$ 
\item The space $H^{1}_{\#}(Y;\text{H}\h^1)$ is defined as the set $\{\phi=(\phi_1,\phi_2)\}$ of all measurable sections of H$\h^1$ such that $\phi \in (H^{1}_{\#,\h}(Y))^2.$  
\item The space $V^{\text{div}}_{\#,\h}(Y)$ is the completion of $\{u\in C^{\infty}_{\#,\h}(Y;\text{H}\h^1) \}$ with respect to the following norm
$$\|u\|_{V^{\text{div}}_{\#,\h}(Y)}=\|u\|_{L^2_{\#,\h}(Y;\text{H}\h^1)}+\|\text{div}_{\h}u\|_{L^2_{\#,\h}(Y)}$$
\end{enumerate}
Now, we will state a version of Theorem 2.16 from \cite{FM02}, which will be used in our analysis.

\begin{thm}\label{div-perp}
Let $F\in L^2(\Omega;V_{\#,\h}^{\emph{div}}(Y)^*)$ such that $$\int_{\Omega}\langle F(x),\phi(x,\cdot)\rangle_{V_{\#,\h}^{\emph{div}}(Y)^*,V_{\#,\h}^{\emph{div}}(Y)} dx=0$$
for all $\phi(x,y)\in L^2(\Omega;V_{\#,\h}^{\emph{div}}(Y))$ with $\emph{div}_{\h,y}\phi=0$ for a.e $x\in \Omega.$ Then, $F=\n_{\h,y}\psi$, with $\psi \in L^2(\Omega;L_{\#,\h}^2(Y)/\R).$
\end{thm}

\section{Definition and properties of unfolding operator}\label{defofunfol}
It has been proved in \cite{MUNA97} for $Y=[-1,1)^3$ that there is a canonical tiling of $\h^1$ associated with the structure of $\h^1$ as a group with dilations,  defined as follows,
\noindent \begin{defn}
Let $\varepsilon>0$ be fixed. Let a typical point of $\mathbb{Z}^3$ is denoted by $k=(k_1,k_2,k_3).$ Define $Y_{k}^\varepsilon=\delta_{\varepsilon}(2k\cdot Y).$
Then $\{Y_k^{\varepsilon}:~k\in \mathbb{Z}^3\}$ is tiling a of 
$\h^1,$ i.e.,\\

(i) $\ds Y_k^{\varepsilon}\cap Y_h^{\varepsilon}=\phi$ if $k\neq h$\\

(ii) $\ds \h^1=\ds \bigcup_{k\in \mathbb{Z}^3}\delta_{\varepsilon}(2k\cdot Y).$
\end{defn} 
The above tilling also holds for $Y=[0,2)^3$. We will use  the tiling of $\Omega$ with $Y=[0,2)^3.$
A slightly modified definition of greatest integer function will be used.
Let $x \in \h^1$, then  $x\in 2k \cdot Y$ for some $k \in \mathbb{Z}^3.$ So, for some $y\in Y$, we can write $x=2k \cdot y.$ Hence, we have $$x_1=2k_1+y_1, ~x_2=2k_2+y_2, ~x_3=2k_3+y_3+4(k_2y_1-k_1y_2).$$ It shows that $2k_1$, $2k_2,$ and $2k_3$ are the greatest even integer less than $x_1,x_2$ and \\$\ds (x_3-4k_2y_1+4y_1k_2).$ This leads us to define greatest even integer function. 
For any $r\in \R$ define  $[r]_e$= greatest even integers less than or equal to $r$. Analogous definition of even fractional part is $\{r\}_e=r-[r]_e.$
 For any $x\in \h^1$, define $$[x]_\h=\frac{1}{2}\left([x_1]_e,[x_2]_e,[x_3-2([x_2]_e\{x_1\}-[x_1]_e\{x_2\})]_e\right).$$ The fractional part of $x$ in $\h^1$, is defined by, \begin{align*} &\{x\}_\h= 2[x]^{-1}_\h \cdot x\\ &=(x_1-[x_1]_e,x_2-[x_2]_e,x_3-[x_3-2([x_2]_e\{x_1\}_e-[x_1]_e\{x_2\}_e)]_e-2([x_2]_e x_1-[x_1]_e x_2))\\
 &=\left(\{x_1\}_e,\,\{x_2\}_e,\, x_3-[x_3-2([x_2]_e\{x_1\}_e-[x_1]_e\{x_2\}_e)]_e-2([x_2]_e x_1-[x_1]_e x_2) \right)
\end{align*}
Now, note that $x_i=[x_i]_e+\{x_i\}_e$ for $i=1,2.$  Using this identity, we have \begin{align*}
2([x_2]_e x_1-[x_1]_e x_2)&=2([x_2]_e[x_1]_e+[x_2]_e\{x_1\}_e)-[x_1]_e[x_2]_e-[x_1]_e\{x_2\}_e\\ &=2([x_2]_e\{x_1\}_e-[x_1]_e\{x_2\}_e)
\end{align*}
Hence, for $x\in \h^1$, the definition for the fractional part can be rewritten as$$\{x\}_\h=(\{x_1\}_e,\,\{x_2\}_e,\, x_3-[x_3-2([x_2]_e\{x_1\}_e-[x_1]_e\{x_2\}_e)]_e-2([x_2]_e\{x_1\}_e-[x_1]_e\{x_2\}_e))$$ 
 Now for any $x\in \delta_{\varepsilon}(2k\cdot Y),$ we can recover $k$ from $x$ as $$k=(k_1,k_2,k_3)=\frac{1}{2}\left(\left[\frac{x_1}{\varepsilon}\right]_e,\left[\frac{x_2}{\varepsilon}\right]_e,\left[\frac{x_3}{\varepsilon^2}-2[x_2/\varepsilon]_e\{x_1/\varepsilon\}_e+2[x_1/\varepsilon]_e\{x_2/\varepsilon\}_e\right]_e\right)$$
Hence using the definition of $[~]_\h$ and $\{\}_\h$, we can write $$k=\left[\delta_{\frac{1}{\varepsilon}}x\right]_\h~\text{ and} ~x=2 \delta_{\varepsilon}\left[ \delta_{\frac{1}{\varepsilon}}x\right]_\h\cdot \delta _{\varepsilon} \left\{\delta_{\frac{1}{\varepsilon}} x\right\}_\h.$$

\noindent Let $\varepsilon>0$, and $\Omega \subset \h^1$ is a  bounded domain. Let $E_{\varepsilon}=\{k\in \mathbb{Z}^3:~Y_k^{\varepsilon}\subset \Omega\} $, $\Omega_\varepsilon=\ds\bigcup_{k\in E_{\varepsilon}}Y_{k}^\varepsilon,$ and $\Lambda_\varepsilon=\Omega\backslash\Omega_\varepsilon.$ Now with the above notations, we are in a position to define unfolding operator in the context of Heisenberg group.
\noindent \begin{defn}
Let $\varepsilon>0,$ then the $\varepsilon$-unfolding of a function $\phi:\Omega\to \R$ is the function $T^{\varepsilon}\phi:\Omega\times Y \to \R$  which defined as
\begin{align*}
T^{\varepsilon}(\phi)(x,y)=
\begin{cases}
&\phi\left( \delta_{\varepsilon}\left(2\left[\delta_{\frac{1}{\varepsilon}}x\right]_\h\right)\cdot \delta_{\varepsilon}y  \right)~~for~(x,y) \in \Omega_ \varepsilon \times Y,\\
 &0~~~~~~~for~~(x,y) \in \Lambda_{\varepsilon}\times  Y. 
\end{cases}
\end{align*}
The operator $T{\varepsilon}$ is called unfolding operator.
\end{defn}
Now we will see some important properties of $T^{\varepsilon}$ in following propositions.
\begin{prop}
Let the unfolding operator $T^{\varepsilon}$ defined as above, then $T^{\varepsilon}$ is linear and for $\phi_1,\phi_2 :\Omega\to \R$, $T^\varepsilon (\phi_1\phi_2)=T^{\varepsilon}(\phi_1)T^{\varepsilon}(\phi_2)$.  
\end{prop}
This follows directly from the definition.
The first important result  to be proved is an $L^1$ integral identity.

\begin{prop}
Let $\phi\in L^1(\Omega).$ Then,  $$\ds \int_{\Omega_\varepsilon}\phi \, dx=\frac{1}{|Y|}\int_{\Omega\times Y}T^{\varepsilon}(\phi)\, dxdy,$$.
\end{prop}
  
\begin{proof}
\begin{align*}
\frac{1}{|Y|}\int_{\Omega \times Y}T^{\varepsilon}(\phi) \, dxdy&=\frac{1}{|Y|}\int_{\Omega}\int_{Y}\phi\left( \delta_{\varepsilon}\left(2\left[\delta_{\frac{1}{\varepsilon}}x\right]_\h \right)\cdot \delta_{\varepsilon}y  \right)\, dxdy\\
&=\ds \sum_{k\in E_{\varepsilon}}\frac{1}{|Y|}\int_{Y_k^{\varepsilon}}\int_{Y}\phi(\delta_{\varepsilon}(2k)\cdot \delta_{\varepsilon} y)\, dxdy\\
&=\ds \sum_{k\in E_{\varepsilon}}\frac{1}{|Y|}\int_{Y}\phi(\delta_{\varepsilon}(2k)\cdot \delta_{\varepsilon} y)|Y_{k}^{\varepsilon}|dy\\
&=\ds \sum_{k\in E_{\varepsilon}}\frac{1}{|Y|}\int_{Y}\phi(\delta_{\varepsilon}(2k)\cdot \delta_{\varepsilon} y)\varepsilon^{4}|Y| dy
\end{align*}
We make the following change of variable as $$z_1=\varepsilon (2k_1+y_1), \; z_2=\varepsilon(2k_2+y_2),\; z_3=\varepsilon^2(2k_3+y_3+4(k_2 y_1- k_1 y_2)). $$ We have $dz=\varepsilon^4 dy.$
 By applying the above change of variable, we get the following equality
  \begin{align*}
  \frac{1}{|Y|}\int_{\Omega \times Y}T^{\varepsilon}(\phi) \, dxdy=\sum_{k\in E_{\varepsilon}}\int_{Y_k^\varepsilon}\phi(z) dz=\int_{\Omega_\varepsilon}\phi.
  \end{align*}
\end{proof}
\noindent   The above integral identity gives us the following proposition.
 \begin{prop}
 \begin{enumerate}
 
  \item For $p\in (1,\infty)$ the operator $T^{\varepsilon}$ is linear continuous map from $L^p(\Omega)$ to $L^p (\Omega\times Y).$
   \vspace{3mm}
  \item $\ds \frac{1}{|Y|}\int_{\Omega\times Y} T^{\varepsilon}(\phi)\, dxdy=\int_{\Omega}\phi \, dx-\int_{\Lambda_\varepsilon} \phi \, dx=\int_{\Omega_\varepsilon}\phi \, dx$
  \vspace{3mm}
  \item $\left|\int_{\Omega}\phi \, dx- \frac{1}{|Y|}\int_{\Omega\times Y}T^{\varepsilon}\phi\right|\leqslant \int_{\Lambda_\varepsilon}|\phi|\, dx $
  \vspace{3mm}
    \item $\|T^{\varepsilon}(\phi)\|_{L^p(\Omega\times Y)}\leqslant |Y|^{\frac{1}{p}}\|\phi\|_{L^p(\Omega)}$.
   \end{enumerate}
   \end{prop} 
   \vspace{1cm}
\noindent   Here, we are considering the domain as a bounded open subset of $\h^1.$ Since the Hausdorff dimension of $\h^1$ is $4$  implies  that the cardinality of the set $$\{k\in \mathbb{Z}^3 :Y_{k}^{\varepsilon}\cap\p \Omega \; \text{is non-empty}\} $$ is $O\left(\frac{1}{\varepsilon^3}\right)$. Hence  $|\Lambda_{\varepsilon}|=O(\varepsilon)$. Also $\chi_{\Lambda_{\varepsilon}}\to 0$ point wise  as $\varepsilon \to 0.$ Hence, we get the following proposition.
  \begin{prop}\label{uci}
Let $u_{\varepsilon}$ be a bounded sequence in $L^2(\Omega)$ with $p\in (1,\infty)$ and $v\in L^{q}(\Omega)$, with $\frac{1}{p}+\frac{1}{q}=1,$
then
$\int_{\Lambda_\varepsilon}u_{\varepsilon}v\to 0$ as $\varepsilon \to 0.$
  \end{prop}
\begin{proof}
Observe that $\chi_{\Lambda_{\varepsilon}}\to 0 $ as $\varepsilon\to 0$. Using the Lebesgue
dominated convergence theorem, we get $ \int_{\Omega} \chi_{\Lambda_{\varepsilon}}|v|^q \, dx \to 0$, and then by using the Holder's inequality we have $\int_{\Lambda_{\varepsilon}} u_{\varepsilon}v\to 0.$  
\end{proof}
\noindent \begin{rem}
In the theory the homogenization or multiscale analysis  the final goal is to pass to the limit $\varepsilon\to 0$. If functions in some integral satisfy the hypothesis of the Proposition \ref{uci}; say for example, $u_{\varepsilon}$ and $v$ are as in the Proposition \ref{uci}, then, we use the following convention,
$$\int_{\Omega}u_{\varepsilon}v=\frac{1}{|Y|}\int_{\Omega\times Y}T^{\varepsilon}(u_{\varepsilon})T^{\varepsilon}(\phi) $$
that is instead of writing approximately equal, we choose to write equality since  at the end,  we will pass to the limit $\varepsilon \to 0.$   
\end{rem}

 \begin{lem}\label{linftyconv} Let $\phi \in C_c^{\infty}(\Omega).$ Then, $\|T^{\varepsilon}(\phi)-\phi\|_{\infty}\to 0$ in $\Omega\times Y.$
  \end{lem}
  \begin{proof} Since $\phi$ is a compactly supported smooth function, it is Lipschitz say with Lipschitz constant $L$.
  Consider \begin{align*}
  |T^{\varepsilon}(\phi)(x,y)-\phi(x)|&=\phi\left(2\delta_\varepsilon \left[\delta_{\frac{1}{\varepsilon}}x\right]_{\h}\cdot \delta_{\varepsilon} y\right) -\phi(x)\\
  &\leq L\left |x- 2\delta_\varepsilon \left[\delta_{\frac{1}{\varepsilon}}x\right]_\h\cdot \delta_{\varepsilon} y\right|_{\R^3}\\
  &\leq C\,  d\left(2\delta_{\varepsilon}\left[\delta_{\frac{1}{\varepsilon}}x\right]_{\h}\cdot \delta_{\varepsilon} \left\{\delta_{\frac{1}{\varepsilon}}x\right\},\, 2\delta_\varepsilon \left[\delta_{\frac{1}{\varepsilon}}x\right]_\h\cdot \delta_{\varepsilon} y\right)\\
&\leq  d\left(\delta_{\varepsilon}\left\{\delta_{\frac{1}{\varepsilon}}x\right\},\, \delta_{\varepsilon} y\right)  
  \leq C \varepsilon.
  \end{align*}
 
The last two inequalities follows from Proposition \ref{equ-heidis}. Passing to the limit $\varepsilon \to 0$ gives the desired result. 
 \end{proof}
 Since $C_c^{\infty}(\Omega)$ is dense in $L^2(\Omega),$ above lemma leads to the following;
\begin{lem}\label{strong-fix-conv}
For $v\in L^{2}(\Omega),$ we have $T^{\varepsilon}(v)\to v$ strongly in $L^2(\Omega\times Y).$
\end{lem}

\noindent Now, we recall the definition of two-scale convergence given in \cite{FM02} for the Heisenberg group.
\begin{defn}
A family of function $\{u_{\varepsilon}\}\in L^2(\Omega)$ is said to be two-scale converges in $\h^1$ to $u_0\in L^2(\Omega\times Y)$, if for any $\psi\in C_c^{\infty}(\Omega;C_{\#.\h}^{\infty}(Y))$, we have 
\begin{align}
\lim_{\varepsilon\to 0}\int_{\Omega}u_\varepsilon(x)\psi\left(x,\delta_{\frac{1}{\varepsilon}}(x)\right)\, dx
=\frac{1}{|Y|}\int_{\Omega\times Y}u_0(x,y)\psi(x,y)\, dxdy.
\end{align}  
\end{defn}     
 \noindent We have already discussed in the introduction that two scale convergence of a sequence  in $L^{2}(\Omega)$ is equivalent to weak convergence of the unfolded sequence in $L^2(\Omega\times Y)$. This result also holds  in Heisenberg group, which is stated in the following proposition,
 
\begin{prop}
Let $v_{\varepsilon}$ be a bounded sequence in $L^2(\Omega)$. Then the following statements are equivalent,
\begin{enumerate}
\item $v_{\varepsilon}$ two-scale converges in $\h^1$ to $v_0\in L^2(\Omega\times Y)$.\\

\item $T^{\varepsilon}(v_{\varepsilon})$ weakly converges to $v_0\in L^2(\Omega\times Y).$
\end{enumerate}
\end{prop}     
 \begin{proof}
 The proof is based on the Lemma \ref{linftyconv}. For $\phi \in C_c^{\infty}(\Omega;C_{\#,\h}^{\infty}(Y))$, for $\varepsilon>0$ small enough, consider the following
\begin{align}\label{equalityoftwoandun}
 \int_{\Omega}v_\varepsilon (x)\phi\left(x,\delta_{\frac{1}{\varepsilon}}(x)\right)=\int_{\Omega \times Y} T^{\varepsilon}(v_\varepsilon)\phi \left(\delta_\varepsilon \left(2\left[\delta_{\frac{1}{\varepsilon}}x \right]_{\h}\right) \cdot \delta_\varepsilon (y),y  \right)+o(1)
.\end{align}  Let $v_{\varepsilon}$ to scale converges to $v_0$ in $\h^1$ and $T^{\varepsilon}\rightharpoonup \hat{v}_0$ weakly in $L^2(\Omega\times Y)$. 
By passing to the $\varepsilon \to 0$ on both side of \eqref{equalityoftwoandun}, we get,
$$\frac{1}{|Y|}\int_{\Omega\times Y} v_{0}(x,y)\phi(x,y)\, dxdy=\frac{1}{|Y|}\int_{\Omega\times Y}\hat{v}_0(x,y)\phi(x,y).$$
As $\phi \in C_c^{\infty}(\Omega;C_{\#,\h}^{\infty}Y)$ is arbitrary, implies $v_{0}(x,y)=\hat{v}_0(x,y)$ a.e. in $\Omega\times Y.$ 

\end{proof}

\subsection{Averaging and  adjoint operators:}

Let $u\in L^p(\Omega)$ and $v\in L^{q}(\Omega\times Y).$ Then, we compute 
\begin{align*}
\frac{1}{|Y|}\int_{\Omega\times Y} T^{\varepsilon}(u)(x,y) v(x,y)= &
  \frac{1}{|Y|}\int_{\Omega\times Y} u\left(2\delta_\varepsilon \left[\delta_{\frac{1}{\varepsilon}}x\right]_\h\cdot \delta_\varepsilon y\right) v(x,y) \, dxdy\\
  & \sum_{k\in E_{\varepsilon}} \frac{1}{|Y|}\int_{Y_k^\varepsilon}\int_{Y} u\left(2\delta_\varepsilon \left[\delta_{\frac{1}{\varepsilon}}x\right]_\h\cdot \delta_\varepsilon y\right) v(x,y) \, dxdy\\
  =&\sum_{k\in E_{\varepsilon}}\int_{Y}\int_{Y}u((\delta_{\varepsilon}(2k)\cdot \delta_\varepsilon y))v(\delta_\varepsilon(2k)\cdot \delta_\varepsilon z,y)\varepsilon^4 dzdy\\   
   \end{align*} 
   Applying the following change of variable $$x_1=\varepsilon(2k_1+y_1),\; x_2=\varepsilon (2k_2+y_2),\; x_3=\varepsilon^2 (2k_3+y_3+4(k_2 y_1-y_2k_1)),$$ we obtain
   \begin{align}\label{adjointofunfolding}
   \begin{split}
   & \frac{1}{|Y|}\int_{\Omega\times Y} T^{\varepsilon}(u)(x,y) v(x,y)
   \\&=\sum_{k\in E_\varepsilon}\frac{1}{|Y|}\int_{\delta_\varepsilon (2k\cdot Y)}\int_{Y}u(x)v(\delta_\varepsilon(2k)\cdot \delta_\varepsilon z,\{\delta_{\frac{1}{\varepsilon}}x\}_\h)\, dxdz\\
   &=\sum_{k\in E_\varepsilon}\int_{\delta_{\varepsilon}(2k\cdot Y)}u(x)\left(\frac{1}{|Y|}\int_{Y}v\left(2\delta_\varepsilon\left[\delta_{\frac{1}{\varepsilon}}x\right]_\h \cdot \delta_\varepsilon z,\{\delta_{\frac{1}{\varepsilon}}(x)\}_\h\right)dz\right)\, dx\\
    &=\int_{\Omega}u(x)\left(\frac{1}{|Y|}\int_{Y}v\left(2\delta_\varepsilon\left[\delta_{\frac{1}{\varepsilon}}x\right]_\h \cdot \delta_\varepsilon z,\{\delta_{\frac{1}{\varepsilon}}(x)\}_\h\right)dz\right)\, dx.
    \end{split}
   \end{align}
  This motivates the following definition
 \begin{defn}
For $p\in L^p(\Omega)$, the averaging operator $\mathcal{U}_\varepsilon:L^{p}(\Omega\times Y)\to L^p(\Omega)$ is defined as

$$\mathcal{U}_\varepsilon(\phi)(x)=\ds\begin{cases}\ds &\frac{1}{|Y|}\int_{Y}\phi\left(2\delta_\varepsilon\left[\delta_{\frac{1}{\varepsilon}}x\right]_\h \cdot \delta_\varepsilon z,\{\delta_{\frac{1}{\varepsilon}}(x)\}_\h\right)dz~~a.e. ~for~x\in \Omega_\varepsilon\\
&0~~~~~~~~~~~~~~~~~~~~~~~~~a.e. ~x\in \Lambda_\varepsilon 
\end{cases}$$ 
\end{defn}
 Using the above definition of $\mathcal{U}_\varepsilon$ in \eqref{adjointofunfolding}, we have; for $\psi \in L^{p}(\Omega)$ and $\phi \in L^q(\Omega \times Y)$, $$\int_{\Omega}\mathcal{U}_{\varepsilon}(\phi)(x)\psi(x)=\frac{1}{|Y|}\int_{\Omega\times Y}\phi(x,y)T^{\varepsilon}(\psi)(x,y).$$
 Hence, this implies,  the adjoint operator of $T^{\varepsilon}$ is $\mathcal{U}_\varepsilon$ in the above sense.   
 Now, we will see how unfolding operator behave with gradient.

\subsection{Unfolding of the gradient:} Throughout this article, we will denote $\n_{\h}$ and $\n_{\h,y}$,  the gradient with respect to $x$ and $y$ respectively on the Heisenberg group. Now, we will see the relation between $\n_{\h} $ and $\n_{\h,y}.$        
 Recall the horizontal vector fields
 \begin{align*}
& X_1=\frac{\p }{\p x_1}+2x_2\frac{\p }{\p x_3},\quad
~X_2=\frac{\p}{\p x_2}-2x_1\frac{\p}{\p x_3}\\
& Y_1=\frac{\p }{\p y_1}+2y_2\frac{\p }{\p y_3},\quad
~Y_2=\frac{\p}{\p y_2}-2y_1\frac{\p}{\p y_3}.
\end{align*}  
Let $\phi\in H^{1}_{\h}(\Omega)$ and let $x\in Y_k^{\varepsilon}.$ Then $$T^{\varepsilon}(\phi)(x,y)=\phi(\varepsilon(2k_1+y_1),\varepsilon(2k_2+y_2),\varepsilon^2(2k_3+y_3+4(k_2y_1-k_1y_2))=\phi(\delta_\varepsilon (2k\cdot y)).$$
By applying the horizontal vector filed $Y_1$ on $T^\varepsilon(\phi),$ we get
\begin{align*}
&Y_1(T^\varepsilon(\phi)(x,y))=\left(\frac{\p }{\p y_1}+2y_2\frac{\p }{\p y_3}\right)(\phi(\delta_\varepsilon (2k\cdot y)))\\&=\varepsilon\frac{\p \phi}{\p x_1}(\delta_\varepsilon(2k\cdot y))+2\varepsilon^{2}(2k_2+y_2)\frac{\p \phi}{\p x_3}(\delta_{\varepsilon}(2k\cdot y))\\
&=\varepsilon T^{\varepsilon}\left(\frac{\p \phi}{\p x_1} \right)+\varepsilon T^{\varepsilon}(x_2) T^{\varepsilon}\left(\frac{\p \phi}{\p x_3} \right)(x,y)       =\varepsilon T^{\varepsilon}(X_1\phi)(x,y)
\end{align*}   
Similarly, we have $Y_2(T^{\varepsilon}(\phi)(x,y))=\varepsilon T^{\varepsilon}(X_2(\phi)$. Hence using these two relations, we get $$\n_{\h,y}(T^{\varepsilon}(\phi)(x,y))=\varepsilon T^{\varepsilon}(\n_{\h}\phi)(x,y).$$ Similarly, $\text{div}_\h$ and $\text{div}_{\h,y}$ denote the divergence with respect to $x$ and $y$ respectively and have the following relation;
$$\text{div}_{\h,y}(T^{\varepsilon}(\phi)(x,y))=\varepsilon T^{\varepsilon}(\text{div}_{\h}\phi)(x,y).$$      
\begin{thm}
Let $u_\varepsilon$ be a sequence in $H_{\h}^1(\Omega)$ such that $u_{\varepsilon}\rightharpoonup u$ weakly in $H_\h^{1}(\Omega)$. Then 
There exist a unique $u_1\in L^{2}(\Omega;H^{1}_{\#,\h}(Y)/{\R})$ such that
\begin{enumerate}\label{gradientconvergence}
\item $T^{\varepsilon}(u_\varepsilon)\to u$ strongly in $L^{2}(\Omega\times Y)$\\
\item $T^{\varepsilon}(\n_{\h}u_\varepsilon)\rightharpoonup \n_\h u+\n _{\h,y}u_1$ weakly in $(L^2(\Omega\times Y))^2$.
\end{enumerate}
\end{thm}
\begin{proof}
First, we will show that the limit of $T^{\varepsilon}(u_\varepsilon)$ is independent of $y$. Let $T^{\varepsilon}(u_\varepsilon)\rightharpoonup \hat{u}$ weakly $L^2(\Omega\times Y)$ and we need to show that $\hat{u}(x,y)=\hat{u}(x)$. To see this, for $\psi\in C_{c}^{\infty}(\Omega\times Y)),$ consider the following
\begin{align}\label{indipendentofy}
\begin{split}
&\int_{\Omega\times Y}\varepsilon T^{\varepsilon}(\n_{\h}u_{\varepsilon})\psi(x,y)\, dxdy\\ &=\int_{\Omega\times Y} \n_{\h,y}(T^{\varepsilon}(u_{\varepsilon})(x,y))\psi(x,y)=-\int_{\Omega\times Y}T^{\varepsilon}(u_\varepsilon)(x,y)\text{div}_{\h,y}\psi(x,y) 
\end{split}
\end{align} 
As $T^{\varepsilon}(\n_{\h} u_{\varepsilon})$ is bounded in $ L^{2}(\Omega\times Y),$ we have $\int_{\Omega\times Y} \varepsilon T^{\varepsilon}(\n_{\h}u_{\varepsilon})\psi(x,y)\to 0$ as $\varepsilon \to 0.$ By letting $\varepsilon \to 0$ in \eqref{indipendentofy} to get
 $$\int_{\Omega\times Y}\hat{u}(x,y)\text{div}_{\h,y}\psi(x,y)=0$$ for all $\psi\in C_{c}^{\infty}(\Omega\times Y).$ Hence $\n_{\h,y}\hat{u}(x,y)=0,$ implies that $\hat{u}$ is independent of $y.$ \\
 
 \noindent On the other hand by weak convergence of $u_{\varepsilon}$, we have $u=M_{Y}(\hat{u})=\hat{u}$, where $M_Y(\hat{u})=\frac{1}{|Y|}\int_{Y}\hat{u}(x,y)dy$. So, we have the weak convergence of $T^{\varepsilon}(u_{\varepsilon})\rightharpoonup u$ in $L^{2}(\Omega \times Y)$. Now for the norm convergence, consider
 \begin{align*}
 \int_{\Omega\times Y}(T^{\varepsilon}(u_{\varepsilon})-T^{\varepsilon}(u))^2\leq \int_{\Omega}(u_{\varepsilon}-u)^2\to 0~as~\varepsilon \to 0.
 \end{align*}
 We know  from Lemma \ref{strong-fix-conv}, that $\|T^{\varepsilon}u\|_{L^2(\Omega\times Y)}\to \|u\|_{L^{2}(\Omega\times Y)}.$ Thus, we have $$\|T^{\varepsilon}(u_{\varepsilon})\|_{L^2(\Omega\times Y)}\to \|u\|_{L^2(\Omega\times Y)}.$$ Hence weak convergence with norm convergence implies the strong convergence. This proves $(1)$ of Theorem \ref{gradientconvergence}. 
 \par
\noindent For the second part, we will use the  test  function of the form $ \psi_{\varepsilon}(x)=\psi\left(x,\delta_{\frac{1}{\varepsilon}}(x)\right)$ for $\psi\in (C_{c}^{\infty}(\Omega,C^{\infty}_{\#,\h}( Y)))^2$ with $\text{div}_{\h,y}\psi=0.$ Let us consider the following,
\begin{align*}
\int_{\Omega} \n_{\h}u_{\varepsilon} \psi\left(x,\delta_{\frac{1}{\varepsilon}}(x)\right)=\int_{\Omega \times Y} T^{\varepsilon}(\n_{\h}u_{\varepsilon})(x,y)\psi\left(\delta_{\varepsilon}\left(2\left[ \delta_{\frac{1}{\varepsilon}}(x)\right]_{\h}\cdot y\right),y\right).
\end{align*}
Let $T^{\varepsilon}(\n_\h u_\varepsilon)\rightharpoonup \xi_0$ weakly in $(L^2(\Omega\times Y))^2$. Now, using integration by parts and  the gradient relation between $\n_{\h}$ and $\n_{\h,y},$ we get
 
\begin{align*}
&\int_{\Omega \times Y} T^{\varepsilon}(\n_{\h}u_{\varepsilon})(x,y)\psi\left(\delta_{\varepsilon}\left(2\left[ \delta_{\frac{1}{\varepsilon}}(x)\right]_{\h}\cdot y\right),y\right)\\ 
&=\int_{\Omega \times Y}\frac{1}{\varepsilon}\n_{\h,y} T^{\varepsilon}(u_{\varepsilon})(x,y)\psi\left(\delta_{\varepsilon}\left(2\left[ \delta_{\frac{1}{\varepsilon}}(x)\right]_{\h}\cdot y\right),y\right)\\
&=-\frac{1}{|Y|}\int_{\Omega \times Y} T^{\varepsilon}(u_{\varepsilon})(x,y)\left[ \text{div}_\h\psi\left(\delta_{\varepsilon}\left(2\left[ \delta_{\frac{1}{\varepsilon}}(x)\right]_{\h}\cdot y\right),y\right)\right]\\
&=-\frac{1}{|Y|}\int_{\Omega \times Y} T^{\varepsilon}(u_{\varepsilon})(x,y)T^{\varepsilon}(\text{div}_{\h}\psi)(x,y)
\end{align*}
By passing to the limit on the both sides, we get
\[\int_{\Omega\times Y} \xi_{0}(x,y) \psi (x,y)=\int_{\Omega\times Y}u(x)\text{div}_\h\psi(x,y).\]
Thus, we have
 $$ \int_{\Omega\times Y} (\xi_0 -\n_{\h} u(x))\psi(x,y)\,dx dy=0.$$  
 From the convergence of the unfolding sequence, we have $\xi_{0},\nabla_{\h}u(x)\in (L^{2}(\Omega\times Y))^2$. Hence $(\xi_{0}(x,y)-\n_{\h}u(x))\in L^{2}\left(\Omega;\left(V_{\#,\h}^{\text{div}}(Y)\right)^*\right)$. Also $(C_{c}^{\infty}(\Omega;C_{\#,\h}^{\infty}(Y)))^2$ is dense in $L^{2}\left(\Omega;\left(V_{\#,\h}^{\text{div}}(Y)\right)\right),$ this implies that $$\int_{\Omega\times Y} (\xi_0 -\n_{\h} u(x))\psi(x,y)\,dx dy=0,~~~\text{for all } \psi \in L^{2}\left(\Omega;\left(V_{\#,\h}^{\text{div}}(Y)\right)\right)~\text{with } \text{div}_{\h}\psi=0.
$$ Hence, $(\xi_0-\n_{\h}u)$ is perpendicular to the divergence free vector field. We get from  Theorem \ref{div-perp} that there exists a unique $u_1\in L^2(\Omega, L^{2}_{\#,\h}(Y)/\R)$ such that $$\xi_0-\n_{\h}u=\n_{\h,y}u_{1}.$$ Since, $\xi_0 $ and $\n_\h u $ are in $(L^{2}(\Omega\times Y))^2$, we see that  $u_{1}\in L^{2}(\Omega;H^{1}_{\#,\h}(Y)/\R).$ Hence, we have the second convergence. 
 \end{proof}

The unfolding $T^{\varepsilon}$  exhibits  more nice properties which will be useful in applications.
\begin{prop} 
Let $u_{\varepsilon}$ be a bounded sequence in $L^{p}(\Omega)$ with $p\in (1,\infty)$ satisfying 
\begin{align*}
\varepsilon\left\|  X_i u_{\varepsilon} \right \|_{L^{p}(\Omega)}\leq C~~\text{ for } i=1,2.
\end{align*}
Then there exists a subsequence and $\hat{u}\in L^p(\Omega)$ with $Y_i\hat{u}\in L^{p}(\Omega\times Y)$ such that 
\begin{align}
\begin{split}
& (i)~T^{\varepsilon}(u_{\varepsilon})\rightharpoonup \hat{u} ~\text{ weakly in } L^{p}(\Omega\times Y )\\
&(ii)~\varepsilon T^{\varepsilon}(X_iu_{\varepsilon})=Y_iT^{\varepsilon}(u_{\varepsilon})\rightharpoonup Y_i \hat{u}~~\text{ weakly in }~ L^{p}(\Omega\times Y)~~\text{for } i=1,2.
\end{split}
\end{align} 
\end{prop} 
 \begin{proof}
 As $u_{\varepsilon}$ is a bounded sequence in $L^{p}(\Omega)$ by properties of unfolding operator, we have $T^{\varepsilon}(u_{\varepsilon})$ is a bounded sequence in $L^{p}(\Omega\times Y).$ Hence by weak compactness, there exists $\hat{u}\in L^{p}(\Omega\times Y)$ such that $$
T^{\varepsilon}(u_{\varepsilon})\rightharpoonup \hat{u}~~\text{ weakly in } L^{p}(\Omega \times Y).$$
Let $\phi \in C_{c}^{\infty}(\Omega\times Y),$ consider
\begin{align*}
\int_{\Omega\times Y} \varepsilon T^{\varepsilon}(X_i u_{\varepsilon})\phi =\int_{\Omega\times Y} Y_iT^{\varepsilon}(u_\varepsilon) \phi=-\int_{\Omega\times Y}T^{\varepsilon}(u_{\varepsilon}) Y_i\phi.
\end{align*} 
Using weak convergence of $T^{\varepsilon}(u_{\varepsilon})$, we pass to the limit as $\varepsilon \to 0$ to get
 $$\lim_{\varepsilon \to 0}\int_{\Omega\times Y} \varepsilon T^{\varepsilon}(X_i u_{\varepsilon})\phi =-\int_{\Omega\times Y}\hat{u} Y_i\phi.$$
 Above equality implies the second part of the proposition.
 \end{proof}
\noindent The above proposition can be written in following form
   \begin{prop} 
Let $u_{\varepsilon}$ be a bounded sequence in $L^{p}(\Omega)$ with $p\in (1,\infty)$ satisfying 
\begin{align*}
\varepsilon\left\| \n_{\h} u_{\varepsilon} \right \|_{L^{p}(\Omega)}\leq C.
\end{align*}
Then, there exist a subsequence and $\hat{u}\in L^p(\Omega;H^1_{\h}(Y))$  such that 
\begin{align}
\begin{split}
& T^{\varepsilon}(u_{\varepsilon})\rightharpoonup \hat{u} ~\text{ weakly in } L^{p}(\Omega; H^1_{\h}(Y)) )\\
&\varepsilon T^{\varepsilon}(\n_{\h}u_{\varepsilon})=\n_{\h, y}T^{\varepsilon}(u_{\varepsilon})\rightharpoonup \n_{\h,y} \hat{u}~~\text{ weakly in }~ (L^{p}(\Omega\times Y))^2.
\end{split}
\end{align} 
\end{prop} 
\noindent Now, let $u_{\varepsilon}$ be a sequence in $H^{1}_\h(\Omega)$ weakly converges to $u$ in $H^{1}_\h(\Omega).$ Then, by compact embedding $u_{\varepsilon}\to u$ strongly in $L^{2}(\Omega).$ By properties of unfolding operator, we have $T^{\varepsilon}(u_{\varepsilon})\to u$ strongly in $L^{2}(\Omega\times Y).$  Thus, we have the following proposition,
\begin{prop}
let $u_{\varepsilon}$ be a sequence in $H^{1}_\h(\Omega)$ weakly converges to $u$ in $H^{1}_\h(\Omega).$ Then $$T^{\varepsilon}(u_\varepsilon)\to u~~\text{weakly in } L^2(\Omega;H^1_\h(Y)),\text{ and strongly in } L^2(\Omega\times Y).$$
\end{prop}

\section{Homogenization via periodic unfolding operator}\label{Homgenization}
In this section, we study the homogenization of a standard oscillation problem in Heisenberg group.  We are under investigation of the applicability of the introduced unfolding operator  in particular to optimal control problems. At this stage we would like to recall that the unfolding operator can be used to characterize the optimal control in homogenization problem (see \cite{AAN19,ARB,AN21}).  

 Let $A=[a_{i,j}]_{i,j=1}^2:\h^1\to M_{2\times 2}(\R)$ be a  matrix valued function with the following properties:  
 \begin{enumerate}
 \item The coefficients  $a_{i,j}:\h^1\to \R$ are Heisenberg $Y$-periodic for all $i,j=1,2,$  bounded and measurable functions.
 \item The matrix $A$ is uniformly elliptic and bounded, that is there exist $\alpha$ and $\beta$ such that following two condition hold
 \begin{enumerate}
 \item $\|A(x)v\|\leq \beta v  $ for all $v\in 
 \R^2$ and for all $x\in\h^1.$ Since $a_{i,j}$ for $i,j=1,2$ are $Y$-periodic, it is sufficient to hold for $x\in Y.$
 \item For all $x\in \h^1$ or $x\in Y$ and $v\in \R^2,$ $A$ satisfies $$\langle A(x)v,v\rangle>\alpha\|v\|^2.$$
 \end{enumerate}
 \end{enumerate}
For each $\varepsilon>0,$ denote $A^{\varepsilon}(x)=A\left( \delta_{\frac{1}\varepsilon}(x) \right).$ The map $x\to A^{\varepsilon}(x)$ can be realized as a moving frame with a section of the vector bundle of symmetric linear endomorphisms of the horizontal fibers.     
As an application of the unfolding operator on the Heisenberg group, we will consider the following homogenization problem: for $f\in L^{2}(\Omega)$, consider
\begin{align}\label{mpde}
\begin{split}
&-\text{div}_{\h}(A^{\varepsilon}\n u_{\varepsilon})+u_{\varepsilon}=f~~\text{in}~ ~~\Omega\\
& \hspace{6mm}   A^{\varepsilon}(x)\n u_{\varepsilon}\cdot n_{\h}(x)=0~~~\text{on}~~ \p\Omega.
\end{split}
\end{align} 
Here $n_\h=C(x)\nu$ where $\nu$ is the Euclidean outward normal on $\partial \Omega.$   
More precisely, we are considering the following  variational problem: find $u_{\varepsilon}\in H^{1}_{\h}(\Omega)$ such that 
\begin{align}\label{varform}
\int_{\Omega}A^{\varepsilon}\n_{\h}u_{\varepsilon}\cdot \nabla_{\h}\phi \, dx+\int_{\Omega}u_{\varepsilon}\phi \, dx=\int_{\Omega}f\phi\, dx, ~~~\text{for all}~\phi\in  H^{1}(\Omega).
\end{align} 
For  every $\varepsilon>0$, Lax-Milgram theorem guaranties  of the  unique solution $u_{\varepsilon}.$ By taking $u_{\varepsilon}$ as a test function on both side of \eqref{varform}, we get $\|u_{\varepsilon}\|_{H^{1}_{\h}(\Omega)}\leq \frac{1}{\alpha} \|f\|_{L^2(\Omega)},$ where $\alpha$ is the elliptic constant constant.
Our goal is to analyze the asymptotic behavior of the sequence of solution $u_{
\varepsilon}$ as the periodic parameter $\varepsilon \to 0.$ 
The present problem is not new and it can be studied via two-scale convergence also. But our aim in this article is to introduce unfolding operator and through this standard example, we are exhibiting the easy way of studying the problem using unfolding operator. We would like to study more non-trivial problem like optimal control problems. 

 The limiting behavior of the  sequence of solution to $u_{\varepsilon}$ is summed up in the following theorem. 
\begin{thm}
Let $u_{\varepsilon}$ be the sequence of solution to \eqref{mpde}. Then 
\begin{align*}
&T^{\varepsilon}(u_\varepsilon)\to u ~\text{ strongly in }  L^2(\Omega\times Y)\\
&T^{\varepsilon}(\n_{\h}u_{\varepsilon})\rightharpoonup \n_{\h}u+\n_{\h,y} u_1  \text{ weakly in }( L^{2}(\Omega\times Y))^2.
\end{align*}
where $u\in H_\h^{1}(\Omega)$ is independent of $y$, and  $(u,u_1)\in H^{1}_{\h}(\Omega)\times L^{2}(\Omega; H^{1}_{\#,\h}(Y)/\R)$ satisfies the following variational system

\begin{align}\label{limitvarform}
\begin{split}
&\int_{\Omega\times Y}A(x,y)(\n_{\h}u(x)+\n_{\h,y}u_1(x,y))\cdot (\n_\h\phi(x)+\n_{\h,y}\phi_1(x,y))\, dxdy\\
&\hspace{3cm}+\int_{\Omega\times}u(x)\phi(x)\, dxdy=\int_{\Omega\times Y}f(x)\phi(x)\, dxdy
\end{split}
\end{align} 
for all $(\phi,\phi_1)\in H^{1}_\h (\Omega)\times L^{2}(\Omega; H^{1}_{\#,\h}(Y)/\R).$ 
\end{thm} 
\begin{proof}
To prove the above theorem, the periodic unfolding operator on the Heisenberg group will be used as the main tool. Since, we have the uniform bound on $\|u_{\varepsilon}\|_{H^{1}_{\h}(\Omega)}$, from Theorem \ref{gradientconvergence},  upto a subsequence we have the existence of $(u,u_1)\in H^{1}_{\h}(\Omega)\times L^{2}(\Omega; H^{1}_{\#,\h}(Y)/\R) $ such that    
\begin{align*} 
&T^{\varepsilon}(u_\varepsilon)\to u ~\text{ strongly in }  L^2(\Omega\times Y)\\
&T^{\varepsilon}(\n_{\h}u_{\varepsilon})\rightharpoonup \n_{\h}u+\n_{\h,y} u_1  \text{ weakly in } L^{2}(\Omega\times Y).
\end{align*}
The proof will be completed if we are able to show that $(u,u_1)$ satisfies the variational form \eqref{limitvarform}. The oscillating test function will be used to prove that $(u,u_1)$ is the solution to the limit variational form. Let $\phi\in C^{\infty}(\bar{\Omega})$ and $\phi_1\in C_{c}^{\infty}(\Omega; C_{\#,\h}^{\infty}(Y)).$ Let $\phi_1^{\varepsilon}(x)=\varepsilon \phi_1\left(x,\delta_{\frac{1}{\varepsilon}}(x)\right)$. Then, we have the following convergence 
\begin{align*}
&T^{\varepsilon}(\phi)\to \phi \text{ strongly in } L^{2}(\Omega\times Y),\\
& T^{\varepsilon}(\phi_1^{\varepsilon}) \to 0 \text{ strongly in } L^{2}(\Omega \times Y).
\end{align*}
Now, by the homogeneous property of the horizontal vector field with respect to the dilation $\delta_{\lambda}$ and the periodicity of $\phi_1$ in $y$, we have 
 $\phi^{\varepsilon}_1(x)= \varepsilon\phi_1\left(x,\delta_{\frac{1}{\varepsilon}}(x)\right)=\varepsilon\phi_1\left(x,\left\{\delta_{\frac{1}{\varepsilon}}(x)\right\}_{\h}\right).$
Now apply $X_1$ on $\phi_1^{\varepsilon},$ we get
\begin{align*}
X_1(\phi^{\varepsilon}_1(x))=& ~\varepsilon\frac{\p\phi_1}{\p x_1}\left(x,\left\{\delta_{\frac{1}{\varepsilon}}(x)\right\}_{\h}\right)+\varepsilon2x_2 \frac{\p \phi_1}{\p x_3}\left(x,\left\{\delta_{\frac{1}{\varepsilon}}(x)\right\}_{\h}\right)\\
&+\frac{\p \phi_1}{\p y_1}\left(x,\left\{\delta_{\frac{1}{\varepsilon}}(x)\right\}_{\h}\right)+2\left\{\frac{x_2}{\varepsilon}\right\}_{e}\frac{\p \phi_1}{\p y_3}\left(x,\left\{\delta_{\frac{1}{\varepsilon}}(x)\right\}_{\h}\right) \\
=&~\varepsilon X_1\phi_1 \left(x,\left\{\delta_{\frac{1}{\varepsilon}}(x)\right\}_{\h}\right)+
Y_1\phi_1\left(x,\left\{\delta_{\frac{1}{\varepsilon}}(x)\right\}_{\h}\right).
\end{align*}
Similarly, we compute  $$X_2(\phi^{\varepsilon}(x))=\varepsilon X_2\left(x,\left\{\delta_{\frac{1}{\varepsilon}}(x)\right\}_{\h}\right)+Y_2 \phi_1\left(x,\left\{\delta_{\frac{1}{\varepsilon}}(x)\right\}_{\h}\right).$$
Combining the above  two equalities, we get the following relation  
$$\n_{\h}(\phi_{1}^{\varepsilon}(x))=\varepsilon\n_{\h}\phi_1\left(x,\delta_{\frac{1}{\varepsilon}}x\right)+\n_ {\h,y}\phi_1\left(x,\delta_{\frac{1}{\varepsilon}}x\right).$$

\noindent Now applying unfolding operator on both sides and passing to the limit, we get  $$T^{\varepsilon}(\n_\h\phi^{\varepsilon}_1)\to \n_{\h,y}\phi_1~\text{ strongly in } L^{2}(\Omega \times Y).$$
 Since $\phi+\phi_1^{\varepsilon}\in H^{1}_\h(\Omega),$ we can use this as a test function in the weak formulation \eqref{varform} to obtain,
 \begin{align*}
\int_{\Omega}A^{\varepsilon}\n_{\h}u_{\varepsilon}\cdot (\nabla_{\h}\phi +\n_{\h}\phi^{\varepsilon}_1 )\, dx+\int_{\Omega}u_{\varepsilon}\phi \, dx=\int_{\Omega}f\phi ~~~\text{for all}~\phi\in H^{1}_\h(\Omega).
 \end{align*}
Applying unfolding operator on both sides of the variational form, we get,

 \begin{align*}
&\int_{\Omega\times}T^{\varepsilon}(A^{\varepsilon}\n_{\h}u_{\varepsilon})(x,y)\cdot T^{\varepsilon}(\nabla_{\h}\phi +\n_{\h}\phi^{\varepsilon}_1 )(x,y)\, dxdy+\int_{\Omega\times Y}T^{\varepsilon}(u_{\varepsilon})(x,y)T^{\varepsilon}(\phi)(x,y) \, dx dy\\&\hspace{9cm}=\int_{\Omega\times Y}T^{\varepsilon}(f)(x,y)T^{\varepsilon}(\phi)(x,y)\, dxdy
 \end{align*}
 $\text{for all}~\phi\in H^{1}_\h(\Omega).$
As $A$ is $Y$-periodic implies that $T^{\varepsilon}(A^{\varepsilon})(x,y)=A(y).$ Using the convergence  of $T^{\varepsilon}(u_{\varepsilon}),~ T^{\varepsilon}(\n_{\h}u_\varepsilon)$ and $T^{\varepsilon}(\phi_{1}^\varepsilon)$, we can pass to the limit as  $\varepsilon \to 0$ in the above integral equality  to obtain

\begin{align}
\begin{split}
&\int_{\Omega\times Y}A(y)(\n_{\h}u(x)+\n_{\h,y}u_1(x,y))\cdot (\n_\h\phi(x)+\n_{\h,y}\phi_1(x,y))\, dxdy\\
&\hspace{3cm}+\int_{\Omega\times Y}u(x)\phi(x)\, dxdy=\int_{\Omega\times Y}f(x)\phi(x)\, dxdy
\end{split}
\end{align} 
for all $(\phi,\phi_1)\in C^{\infty}(\bar{\Omega})\times C_{c}^{\infty}(\Omega;C_{\#,\h}(Y)).$  By density, we have that the above  equality is true  for all $(\phi,\phi_1)\in H^{1}_\h (\Omega)\times L^{2}(\Omega; H^{1}_{\#,\h}(Y)/\R).$ In order to get the convergence of the full sequence it is sufficient to show that the  the limit variational form \eqref{limitvarform} admits unique solution. Uniqueness will be proved if we establish that the following bi-linear form $$B:H^{1}_\h (\Omega)\times L^{2}(\Omega; H^{1}_{\#,\h}(Y)/\R)\times H^{1}_\h (\Omega)\times L^{2}(\Omega; H^{1}_{\#,\h}(Y)/\R) \to \R,$$
given by  
\begin{align*}
&B((u,u_1),(\phi,\phi_1))\\ &= \int_{\Omega\times Y}A(y)(\n_{\h}u(x)+\n_{\h,y}u_1(x,y))\cdot (\n_\h\phi(x)+\n_{\h,y}\phi_1(x,y))\, dxdy+\int_{\Omega\times Y}u\phi \, dxdy
\end{align*} 
is elliptic. The ellipticity of $B$ follows from the eillipticity of $A$. To be more precise 
$$B((\phi,\phi_1),(\phi,\phi_1))>\frac{\alpha}{2}(\|\phi\|^2_{H^{1}(\Omega)}+\|\phi_{1}\|^2_{L^{2}(\Omega;H^{1}_{\#}(Y)/\R )}).$$ Hence this completes the proof the theorem. 
\end{proof}

We can write the variational form \eqref{limitvarform} which is in two-scale  in more explicit way using the cell problem. More precisely, we can get the one-scale form (homogenized equation). In order to write scale separated form, let us put $\phi=0$ in \eqref{limitvarform} to get,

\begin{align}\label{cellvarform}
\begin{split}
&\int_{\Omega\times Y}A(x,y)(\n_{\h}u(x)+\n_{\h,y}u_1(x,y))\cdot \n_{\h,y}\phi_1(x,y)\, dxdy=0
\end{split}
\end{align} 
 Let us introduce the following cell problem; for $i=1,2,$ find $Z_i\in H^1_{\#,\h}(Y)/\R$ such that
 \begin{align}
 \int_{Y}A(y)\n_{\h,y}Z_i(y)\cdot \n_{\h,y}\xi(y) dy=-\int_{Y}A(y)e_i\cdot\n_{\h,y}\xi(y) dy
 \end{align} for all $\xi\in H_{\#,\h}(Y)/\R.$ Here $e_i$ for $i=1,2$ denote the standard basis for $\R^2.$ Using $Z_i$, we can write $u_1(x,y)=\sum_{i=1}^2 Z_i(y) X_i u(x)$. Now put $\phi_1=0$ in the variational form \eqref{limitvarform} and substitute  $u_1(x,y)=\sum_{i=1}^2 Z_i(y) X_i u(x)$ to get 
 
 \begin{align*}
 \int_{\Omega\times Y}A(y)\left(\n_\h u(x)+\sum_{i=1}^2\n_{\h,y}Z_iX_i u\right)\cdot \n_\h \phi \,dxdy+|Y|\int_{\Omega}u \phi \, dx=|Y|\int_{\Omega}f\phi \, dx . 
 \end{align*}
 Above equality can be written as,
\begin{align}\label{scaleseparatedvarform}
 \int_{\Omega}\left(\int_{Y} A(y)(I_{2\times 2}+\begin{bmatrix}
 \n_{\h,y}Z_1~\n_{\h,y}Z_2
 \end{bmatrix} dy  \right)\n_{\h}u\cdot \n_{\h}\phi \, dx+|Y|\int_{\Omega} u\phi \,dx=|Y|\int_{\Omega}f\phi \,dx
\end{align} 
 Denote the homogenized  constant coefficient matrix $$A_0=\int_{Y} A(y)(I_{2\times 2}+\begin{bmatrix}
 \n_{\h,y}Z_1~\n_{\h,y}Z_2
 \end{bmatrix} dy=\int_{Y} A(y)\left(
 \begin{bmatrix}
 0 & 1\\ 1 & 0
 \end{bmatrix}+\begin{bmatrix}
 Y_1Z_1  & Y_1 Z_2\\
 Y_2 Z_1 &Y_2Z_2
 \end{bmatrix}\right) dy.$$
 Hence the variational form \eqref{scaleseparatedvarform}, reduces to    

\begin{align}\label{scalesepvarform} 
\int_{\Omega}A_0\n_{\h}u\cdot \n_{\h}\phi \, dx+|Y|\int_{\Omega} u\phi \,dx=|Y|\int_{\Omega}f\phi \,dx.
\end{align} The above variational form hold for all $\phi\in H_\h^{1}(\Omega).$ Hence the the varional form \eqref{scalesepvarform} corresponds to the following strong form
\begin{align}\label{strongform}
\begin{split}
-\text{div}_\h(A_0\n_\h u) +|Y|u=&|Y|f~~\text{in}~\Omega,\\
A_0\n_\h u\cdot n_\h=&0~~\text{on}~\p \Omega.
\end{split}
\end{align} 
 This is the homogenized system corresponding to \eqref{mpde}.

\subsection{Optimal control problem} To demonstrate the use of unfolding operator,
in this subsection, we will show how unfolding operator helps to characterize  the periodic interior optimal control. Let the admissible control set is $L^{2}_{\#,\h}(Y)$ and for $\theta \in L^{2}_{\#,\h}(Y),$ we denote $\theta^{\varepsilon}(x)=\theta \left(\delta_{\frac{1}{\varepsilon}}(x)\right)=\theta\left(\left\{\delta_{\frac{1}{\varepsilon}}(x)\right\}_\h\right).$ Let $A^{\varepsilon}(x)$, $\Omega,\Omega_\varepsilon$ are as defined earlier. We consider the following $L^{2}$- cost functional,
\begin{align}
J_{\varepsilon}(u_{\varepsilon},\theta)=\frac{1}{2}\int_{\Omega}A^{\varepsilon}\n_\h u_{\varepsilon}\cdot \n_\h u_{\varepsilon}+\frac{\rho}{2}\int_{\Omega_\varepsilon}|\theta^{\varepsilon}|^{2},
\end{align}
where $\beta>0$ is a regularization parameter and $u_{\varepsilon}$ satisfies the following constrained PDE,
\begin{align}\label{mpdeopt}
\begin{split}
&-\text{div}_{\h}(A^{\varepsilon}\n_\h u_{\varepsilon})+u_{\varepsilon}=f+\chi_{\Omega_\varepsilon}\theta^{\varepsilon}~~\text{in}~ ~~\Omega\\
& \hspace{6mm}   A^{\varepsilon}(x)\n_\h u_{\varepsilon}\cdot n_{\h}=0~~~\text{on}~~ \p\Omega,
\end{split}
\end{align}  
with $f\in L^2(\Omega).$  The optimal control problem is to find $(\bar{ u}_{\varepsilon},\bar{\theta}_{\varepsilon})\in H_\h^{1}(\Omega)\times L^{2}_{\#,\h}(Y)$ such that 
\begin{align}\label{opthei}
J_{\varepsilon}(\bar{ u}_{\varepsilon},\bar{\theta}_{\varepsilon})=\inf\{J_{\varepsilon}(u_{\varepsilon},\theta):~(u_{\varepsilon},\theta) ~\text{satisfies } \eqref{mpdeopt}\}
\end{align}
As $A^{\varepsilon}$ is uniformly elliptic and $\rho>0$ imply that $J_{\varepsilon}$ is strictly convex. Hence, the classical method of calculus of variation ensures the existence and uniqueness of $(\bar{u}_{\varepsilon},\bar{\theta}_{\varepsilon}).$  The following theorem theorem gives the characterization of the optimal control in the $\varepsilon$ stage. 
\begin{thm}
Let $(\bar{u}_{\varepsilon},\bar{\theta}_\varepsilon)\in H_\h^{1}(\Omega)\times L^2_{\#,\h}(Y)$ be the optimal solution to the optimal control problem \eqref{opthei}. Then, the optimal control $\bar{\theta}_{\varepsilon}$ can be written as
\begin{align}
\bar{\theta}_{\varepsilon}=\frac{1}{|\Omega|}\int_{\Omega}T^{\varepsilon}(\bar{v}_{\varepsilon})(x,y)\, dx,
\end{align}
where $\bar{v}_{\varepsilon}$ satisfies the following adjoint PDE,
\begin{align}\label{adeq}
\begin{cases}
&\ds -\emph{div}_\h(A^{\varepsilon}\n_\h \bar{v}_{\varepsilon})+ \bar{v}_{\varepsilon}=-\emph{div}_\h\left(A^{\varepsilon}\n_\h \bar{u}_{\varepsilon} \right)~~\text{in} ~~\Omega,\\
&\ds A^{\varepsilon}\n_\h \bar{v}_{\varepsilon}\cdot n_\h=0~~\text{on}~~\partial \Omega.
 \end{cases}
\end{align}
\end{thm} 
\begin{proof}
Given $\theta \in L^2_{\#,\h}(Y),$ denote $F_{\varepsilon}(\theta)=J_{\varepsilon}(u^{\varepsilon}(\theta),\theta)$
where $u_{\varepsilon}(\theta)$ is the solution to \eqref{mpdeopt}. Evaluating the limit of $$\frac{1}{\lambda}(F_{\varepsilon}(\bar{\theta}_{\varepsilon}+\lambda \theta)-F_{\varepsilon}(\bar{\theta}_{\varepsilon}))$$ as $\lambda \to 0$ and denoting the limit by $F'(\bar{\theta}_{\varepsilon})\theta$, we get 
\begin{align*}
F'_{\varepsilon}(\bar{\theta}_\varepsilon) \theta=\int_{\Omega}A^{\varepsilon} \n_\h \bar{u}_{\varepsilon}\cdot \n_\h w_{\theta} \, dx+\rho \int_{\Omega_\varepsilon} \bar{\theta}^{\varepsilon}_{\varepsilon}\theta^{\varepsilon} \, dx.
\end{align*}
where $w_{\theta}$ is the solution to the following PDE,

\begin{align}\label{wtheta}
\begin{cases}
&\ds -\text{div}_\h\left(A^{\varepsilon}\n_\h w_{\theta} \right)+ w_{\theta}=\chi_{\Omega_{\varepsilon}}\theta^{\varepsilon}~~\text{in} ~~\Omega,\\
&\ds  A^{\varepsilon}\n_\h w_{\theta}\cdot n_\h=0~~\text{on}~~\partial \Omega.
 \end{cases}
\end{align}
As $(\bar{u}_{\varepsilon}, \bar{\theta}_{\varepsilon})$ is the optimal solution, we have
$$F'_{\varepsilon}(\bar{\theta}_{\varepsilon})\theta=0,~~\text{for~all} ~~\theta \in L^{2}_{\#,\h}(Y).$$ 
Hence, we get,
\begin{align}\label{beptheta}
\int_{\Omega} A^{\varepsilon}\n_\h \bar{u}_{\varepsilon} \cdot \n_\h w_{\theta}\, dx=-\rho \int_{\Omega_\varepsilon} \bar{\theta}^{\varepsilon}_{\varepsilon} \theta^{\varepsilon}\, dx.
\end{align}
Let $\bar{v}_{\varepsilon}$ satisfies \eqref{adeq}. Using $w_{\theta}$ as a test function in \eqref{adeq} and $\bar{v}_{\varepsilon}$ in \eqref{wtheta}, we obtain

\begin{align}\label{bepv}
\int_{\Omega} A^{\varepsilon}\n_\h \bar{u}_{\varepsilon}\cdot \n_\h w_{\theta}\, dx=\int_{\Omega_\varepsilon} \bar{v}_{\varepsilon}  \theta^{\varepsilon}\, dx.
\end{align} 
Hence from \eqref{beptheta} and \eqref{bepv}, we have 
\begin{align}\label{22}
\int_{\Omega_\varepsilon}\bar{\theta}^{\varepsilon}_{\varepsilon}\theta^{\varepsilon}=-\frac{1}{\rho} \int_{\Omega_\varepsilon}\bar{v}_{\varepsilon}\theta^{\varepsilon}.
\end{align} Now from the definition of unfolding, we have $T^{\varepsilon}(\theta^{\varepsilon})(x,y)=\theta(y)$ and $T^{\varepsilon}(\bar{\theta}^{\varepsilon}_\varepsilon)(x,y)=\bar{\theta}_{\varepsilon}(y)$. Hence, by applying the unfolding operator on both sides of \eqref{22}, we get, $$\int_{\Omega\times Y}\bar{\theta}_\varepsilon(y) \theta(y)=-\frac{1}{\rho}\int_{\Omega\times Y}T^{\varepsilon}(v_{\varepsilon})(x,y) \theta(y).$$ The above equality holds for all $\theta\in L^{2}_{\#,\h}(Y),$ which implies that $$\bar{\theta}_{\varepsilon}(y)=-\frac{1}{\rho|\Omega|}\int_{\Omega}T^{\varepsilon}(\bar{v}_{\varepsilon})(x,y)\, dx$$     This completes the proof.
\end{proof}

\noindent We have only characterized the optimal control problem using the unfolding operator. Indeed, we can study the homogenization of the problem and obtain the limit problem along the similar lines as in the beginning of this section. Hence omit further details.
\begin{rem}
The unfolding operator, we have  defined is not restricted to $\h^1$, it can be extended in the same way to any $\h^{n}.$ Similarly, the problem under consideration can also be studied in any $\h^n$  for $n \in \mathbb{N}$, using the unfolding operator.     
\end{rem}


\begin{thebibliography}{99}

\bibitem{smooth} {\sc S. Aiyappan, A. K. Nandakumaran and R. Prakash},  \emph{Generalization of unfolding operator for highly oscillating smooth boundary domains and homogenization}, Calculus of Variations and Partial Differential Equations. 57.3 (2018), 86. 
 
\bibitem{AAN19}{\sc S. Aiyappan, A. K. Nandakumaran, and Abu Sufian}, \emph{Asymptotic analysis of a boundary optimal control problem on a general branched structure}, Mathematical Methods in the Applied Sciences 42.18 (2019): 6407--6434. 

 \bibitem{semiopt} {\sc S.~Aiyappan, A.~K.~Nandakumaran, and R.~Prakash}, \emph{Semi-Linear Optimal Control Problem on a smooth oscillating domain}, Communications in Contemporary Mathematics (2018).

 
\bibitem{GA92} {\sc G. Allaire},\emph{ Homogenization and two-scale convergence}, SIAM Journal on Mathematical Analysis 23.6 (1992), pp.~1482--1518. 
 
 
\bibitem{MUNA97} {\sc M. Biroli, U. Mosco, and N. A. Tchou}, \emph{Homogenization by the Heisenberg group}, Advances in Mathematical Sciences and Applications 7 (1997), pp.~809--831.
  
\bibitem{MNAZ99}{\sc M. Biroli, N. A. Tchou, V. V. Zhikov},\emph{ Homogenization for Heisenberg operator with Neumann boundary
conditions}, Ric. Mat. XLVIII (1999),pp.~ 45-–59. 
 
 
\bibitem{FM02}{\sc B. Franchi, and M. C. Tesi}, \emph{Two-scale homogenization in the Heisenberg group}, Journal de mathématiques pures et appliquées 81.6 (2002): 495--532.

\bibitem{FM06}{\sc B. Franchi, Nicoletta Tchou, and Maria Carla Tesi}, {\emph Div–curl type theorem, H-convergence and Stokes formula in the Heisenberg group}, Communications in Contemporary Mathematics 8.01 (2006): 67--99.

\bibitem{FCT18}{\sc B. Franchi,  Cristian E. Gutiérrez, and Truyen van Nguyen}, \emph{Homogenization and convergence of correctors in Carnot groups}, Communications in Partial Differential Equations 30.12 (2005): 1817--1841.



\bibitem{priun}
{\sc D.~Cioranescu, A.~Damlamian and G.~Griso}, \emph{The periodic unfolding method in homogenization}, SIAM Journal on Mathematical  Analysis, 40.4 (2008), pp.~1585--1620.


\bibitem{CioDAmGri1}
{\sc D.~Cioranescu, A.~Damlamian and G.~Griso},  \emph{ The periodic unfolding method: Theory and applications to Partial Differential Problems ,} Series in Contemporary Mathematics 03, Springer. 2019.


 \bibitem{ARB} {\sc A. K. Nandakumaran, R. Prakash and B. C. Sardar}, \emph{Homogenization of an optimal control
via unfolding method}, SIAM Journal on Control and Optimization, 53.5 (2015), pp.~ 3245--3269.

\bibitem{JDE}{\sc A.~K. Nandakumaran and A.~Sufian},\emph{Strong contrasting diffusivity in general oscillating
domains: Homogenization of optimal control problems}, Journal of Differential Equations,
https://doi.org/10.1016/j.jde.2021.04.031

\bibitem{AN21}{\sc A. K. Nandakumaran,  and Abu Sufian}, \emph{Oscillating PDE in a rough domain with a curved interface: Homogenization of an Optimal Control Problem}, ESAIM: Control, Optimisation and Calculus of Variations 27 (2021): S4.

\bibitem{GN89}{\sc G. Nguetseng},\emph{A general convergence result for a functional related to the theory of homogenization}, SIAM
Journal on Mathematical Analysis 20 (1989),pp.~ 608-–623.
 \bibitem{zhik96}{\sc V. V. Zhikov},\emph{ Connectedness and homogenization. Examples of fractal conductivity}, Sb. Math. 187 (12)
(1996),pp.~ 1109–-1147.

 \end{thebibliography}
\end{document}